\documentclass[12pt]{amsart}
\usepackage{amsmath,amssymb,amsthm}
\usepackage{ifthen,verbatim}
\usepackage{mathrsfs}
\usepackage{color}
\usepackage{url}
\usepackage{pdfpages}
\usepackage[english]{babel}
\usepackage{verbatim,here}
\usepackage[T1]{fontenc}
\usepackage{floatflt,graphicx,graphics}
\usepackage{a4wide}
\usepackage[urlcolor=blue,colorlinks=true]{hyperref}
\usepackage{url}
\usepackage{tikz}
\usepackage{tabularx}
\usepackage{subfigure}

\numberwithin{equation}{section}
\nonstopmode
\theoremstyle{plain}
\newtheorem{cor}[equation]{Corollary}
\newtheorem{corollary}[equation]{Corollary}

\newtheorem{theorem}[equation]{Theorem}

\newtheorem{lemma}[equation]{Lemma}

\theoremstyle{definition}

\newtheorem{remark}[equation]{Remark}

\newtheorem{nonsec}[equation]{}

{\qed\bigskip}

\newcounter{alphabet}

\newcommand{\be}{\begin{eqnarray}}
\newcommand{\ee}{\end{eqnarray}}
\newcommand{\ba}{\begin{array}}
\newcommand{\ea}{\end{array}}
\newcommand{\ben}{\begin{eqnarray*}}
\newcommand{\een}{\end{eqnarray*}}

\newcommand{\card}{{\operatorname{card}\,}}

\newcommand{\R}{{\mathbb R}}

\newcommand{\B}{\mathbb{B}}

\newcommand{\capa}{\mathrm{cap}\,}

\newcommand{{\tth}}{\mathrm{th}}

\newcommand{\uhp}{{\mathbb H}}

\newcommand{\K}{\mathcal{K}}

\newcommand {\M} {\mathsf{M}}

\font\fFt=eusm10 
\font\fFa=eusm7  
\font\fFp=eusm5  
\def\K{\mathchoice
{\hbox{\,\fFt K}}
{\hbox{\,\fFt K}}
{\hbox{\,\fFa K}}
{\hbox{\,\fFp K}}}

\newcounter{minutes}
\divide\time by 60
\newcounter{hours}
\multiply\time by 60
\addtocounter{minutes}{-\time}

\begin{document}

\bibliographystyle{amsplain}
\title
{
Conformally invariant metrics and lack of H\"older continuity
} 

\def\thefootnote{}
\footnotetext{
\texttt{\tiny File:~\jobname .tex,
          printed: \number\year-\number\month-\number\day,
          \thehours.\ifnum\theminutes<10{0}\fi\theminutes}}
\makeatletter\def\thefootnote{\@arabic\c@footnote}\makeatother

\author[R. Kargar]{Rahim Kargar$^*$}
\address{R. Kargar: Department of Mathematics and Statistics, University of Turku, Turku, Finland}
\email{rakarg@utu.fi\newline ORCID ID:  \url{http://orcid.org/0000-0003-1029-5386}}

\author[O. Rainio]{Oona Rainio}
\address{O. Rainio: Department of Mathematics and Statistics, University of Turku, Turku, Finland}
\email{ormrai@utu.fi\newline ORCID ID: \url{http://orcid.org/0000-0002-7775-7656} }

\keywords{Conformal modulus; Hyperbolic metric; Modulus metric; H\"{o}lder continuity.}
\subjclass[2020]{51M09; 30F45}
\begin{abstract}
The modulus metric between two points in a subdomain of $\mathbb{R}^n, n\ge 2,$ is defined in terms of moduli of curve families joining the boundary of the domain with a continuum connecting the two points. This metric is one of the conformally invariant hyperbolic type metrics that have become a standard tool in geometric function theory. We prove that the modulus metric is not H\"{o}lder continuous with respect to the hyperbolic metric.
\end{abstract}
\maketitle

\section{Introduction}

In the geometric function theory, one of the important subjects of study
is the modulus of a curve family in $\R^2$ introduced in 1950 by
L. Ahlfors and A. Beurling. 
Its definition was extended
to the Euclidean space $\R^n$ with dimension $n\geq2$ by
B. Fuglede, 
after which it was quickly adopted as
a standard tool to study mappings during the early 1960s by
F.W. Gehring \cite{g05} and J. V\"ais\"al\"a \cite{vbook}.
This conformal invariant has numerous applications in the
current research, see \cite{bp, hkv,h, svz}.

Due to its invariance properties, the conformal modulus of a curve
family is often used to study the distortion of distances between
points under quasiconformal mappings. At times, it is enough to use
crude estimates for simple curve families combined with
majorization or minorization. However, the use of crude estimates
has two drawbacks: Firstly, it requires experience with the moduli
of curve families and,
what is more unfortunate, crude estimates lead to loss of information.

Instead of using crude estimates, it is often useful to reduce the
estimation problem to classical extremal problems and to use these
systematically.
We now discuss two classical extremal problems which have been
applied in this way.
Namely, these two problems have already been studied by
H. Gr\"otzsch and O. Teichm\"uller, who both were pioneers of
conformal invariant from 1920s and 1930s on.

Let $G$ be a domain in the extended real space $\overline{\R}^n=\mathbb{\R}^n\cup \{\infty\}$
such that $\card(\overline{\R}^n\backslash G)\geq2$. Our first
extremal problem is \cite[(10.2), p. 174]{hkv}
\begin{align}\label{mumetric}
\mu_G(x,y)=\inf_{C_{xy}}\M(\Delta(C_{xy},\partial G;G)),
\end{align}
where the infimum is taken over all continua $C_{xy}$ joining
the points $x$ and $y$ in $G$, $\Delta(E,F;G)$ is the family of all
such curves in $G$ that join $E$ and $F$, and $\M(\Gamma)$
denotes the conformal modulus of a curve family $\Gamma$ in $G$.
The second extremal problem is
\begin{align}\label{lambdametric}
\lambda_G(x,y)=\inf_{C_x,C_y}\M(\Delta(C_x,C_y;G)),
\end{align}
where infimum is taken over all pairs of continua $C_x$ and $C_y$ in $\overline{G}$
such that $x\in C_x$, $y\in C_y$,
$\overline{C_x}\cap\partial G\neq\varnothing$ and
$\overline{C_y}\cap\partial G\neq\varnothing$.

By Ahlfors \cite[p. 72]{ah}, the extremal problem \eqref{mumetric}
was studied with $n=2$ and $G=\B^2$ by H. Gr\"otzsch while the problem
\eqref{lambdametric} was considered with $n=2$ and $G=\R^2\backslash\{0\}$
by O. Teichm\"uller. For further notes on the literature, see
\cite[10.30, p. 186]{hkv}, where the contributions of I.S. G\'al,
T. Kuusalo and J. Lelong-Ferrand are cited. As stated in the following
two theorems \cite{hkv, vuo, vuo90}, the functions $\mu_G$ and $\lambda_G$
of the extremal problems \eqref{mumetric} and \eqref{lambdametric} can
also be used to define metrics, out of which the first one is called the
\emph{modulus metric}.

\begin{theorem}\label{thm_mulambdametrics}
(1) If $\capa(\partial G)>0$, then $\mu_G$ is a metric.\newline
(2) $\lambda_G^p$ is a metric if and only if $p\in[-1/(n-1),0)$.
\end{theorem}

The proof of Theorem \ref{thm_mulambdametrics}(1) is straightforward, whereas part (2) has an interesting history. It was proved with
$p=-1/n$ by Ferrand \cite{lf} and in the special case $G=\B^2$ by
G.D. Anderson, M.K. Vamanamurthy and M. Vuorinen \cite{avv}.
As an open problem, part (2) was formulated in \cite[p. 193]{vuobook1988}
and J. Ferrand, J.A. Jenkins, and A.Y. Solynin found solutions independently; see \cite[p. 453]{hkv}.

Theorem \ref{thm_mulambdametrics} has numerous applications, based
on the next two theorems.

\begin{theorem}\label{thm_muLamB}
For $D\in\{\uhp^n,\B^n\}$ and all distinct points $x,y\in D$,
\begin{enumerate}
  \item $\mu_D(x,y)=\gamma_n(1/{\rm th}(\rho_D(x,y)/2))$;
  \item $\lambda_D(x,y)=2^{-n}\gamma_n({\rm ch}(\rho_D(x,y)/2)),$
\end{enumerate}
where $\rho$ stands for the hyperbolic metric and $\gamma_n$
is a special function called the Gr\"otzsch capacity.
\end{theorem}

\begin{theorem}\label{thm_Kformulambda}
Let $f:G\rightarrow G'$ be a $K$-quasiconformal homeomorphism. Then for all distinct $x,y\in G$,
\begin{equation*}
\mu_G(x,y)/K\leq\mu_{G'}(f(x),f(y))\leq K\mu_G(x,y)
\end{equation*}
and
\begin{equation*}
\lambda_G(x,y)/K\leq\lambda_{G'}(f(x),f(y))\leq K\lambda_G(x,y).
\end{equation*}
In other words, written as mappings between metric spaces,
\begin{align*}
f:(G,\mu_G)\to(G',\mu_{G'}),\quad\text{or}\quad
f:(G,\lambda_G^{1/(1-n)})\to(G',\lambda_{G'}^{1/(1-n)})
\end{align*}
the mapping $f$ is bi-Lipschitz. For the case of $\mu_G$,
we assume that $\capa(\partial G)>0$ and, for the case of
$\lambda_G$, let $\card(\overline{\R}^n\backslash G)\geq2$.
\end{theorem}

The metrics $\rho_D$, $\mu_D$ and $\lambda_D^{1/(1-n)}$
in $D\in\{\uhp^n,\B^n\}$ are all conformally invariant,
so it is natural to ask whether they are comparable
in some other fashion. Theorem \ref{thm_Kformulambda} implies
that quasiconformal mappings satisfy a version of the Schwarz lemma,
according to which mappings are H\"older continuous with respect to the
hyperbolic and the Euclidean metrics. As well known \cite[Thm 16.3, Thm 16.21]{hkv}
 for dimensions $n\geq3$, both the variants of the Theorem
\ref{thm_Kformulambda} yield different H\"older exponents,
while Theorem \ref{thm_Kformulambda} itself speaks for bi-Lipschitz
continuity in the respective two metric spaces.

The above two theorems show that quasiconformal mappings are bi-Lipschitz
in the respective metric spaces. However, because the local structure of these
spaces depends on the special function $\gamma_n,$ it is desirable to compare
metrics to Euclidean and hyperbolic metrics and this is what we will do here.
Surprisingly, it turns out that the $\mu_{\B^2}$ metric is not H\"older
equivalent to the hyperbolic metric $\rho_{\B^2},$ but nevertheless we can conclude sharp modulus of continuity estimates from the above theorems. We also prove various other results for these metrics.

\begin{theorem}\label{nonHol}
The $\mu_D$ metric, $D\in\{\uhp^2,\B^2\},$ is not H\"older continuous
with respect to $\rho_D.$
\end{theorem}
Here, it should be noted that the same result holds for $\lambda^{-1}_{D}(x,y)$ metric as proven in Corollary \ref{cor-FerNotHol}. For the H\"older non-equivalence of $\mu_D$ and, $\lambda_D^{1/(1-n)}$ see \cite[16.6]{avv}.


Very recently,  $\mu_G$ isometries were
thoroughly studied by D. Betsakos- S. Pouliasis \cite{bp}, X. Zhang \cite{z21},
and  S. Pouliasis- A. Solynin \cite{pso}. In this case, it turns out that
$\mu_G$ isometries are, in fact, conformal mappings.
One might expect that a similar
result also holds for the $\lambda_G^{1/(1-n)}$ metric, but this seems to be an
open problem. Isometries of some other metrics were studied by
P.  H\"ast\"o, Z. Ibragimov, D. Minda,
S. Ponnusamy, and  S. Sahoo \cite {himps}.

The structure of this article is as follows.
In Section \ref{sec2-pre} we give some necessary definitions and notations.
In Section \ref{sec3-ul bounds} we prove that the modulus metric is not H\"{o}lder continuous with respect to the hyperbolic metric.

\section{Preliminaries}\label{sec2-pre}

Let us first introduce our notations. For all points $x\in\R^n$ and any
radius $r>0$, we can define an open Euclidean ball
$B^n(x,r)=\{y\in\R^n\text{ }:\text{ }|x-y|<r\}$ and its boundary sphere
$S^{n-1}(x,r)=\{y\in\R^n\text{ }:\text{ }|x-y|=r\}$.
For the unit ball and sphere, we use the simplified notations $\B^n=B^n(0,1)$
 and $S^{n-1}=S^{n-1}(0,1)$. Denote also the $(n-1)$-dimensional surface area
 of $S^{n-1}$ by $\omega_{n-1}$ and define the following constant $c_n$
 (see Ref. \cite[p. 41]{avv}):
\begin{align*}
c_2=\frac{2}{\pi},\quad
c_n=2^{1-n}\omega_{n-2}
\left(\int^{\pi/2}_0\sin^{\frac{2-n}{n-1}}t\,dt\right)^{1-n}
\text{ for }n\geq 2.
\end{align*}

The \emph{hyperbolic metric} $\rho$ can be defined in the unit ball with the
formula \cite[(4.16), p. 55]{hkv}
\begin{align*}
{\rm sh}^2\left(\frac{\rho_{\B^n}(x,y)}{2}\right)=\frac{|x-y|^2}{(1-|x|^2)(1-|y|^2)},\quad x,y\in\B^n
\end{align*}
and, by the conformal invariance of this metric, we can compute its value
in any such domain that can be mapped conformally onto the unit ball.
A hyperbolic ball $B_{\rho}(x,M)$ defined in the unit ball $\B^n$ is equal to the Euclidean ball $B^n(y,r)$, where \cite[(4.20), p. 56]{hkv}
\begin{align*}
y=\frac{x(1-t^2)}{1-|x|^2t^2},\quad
r=\frac{(1-|x|^2)t}{1-|x|^2t^2},\quad
t={\rm th}(M/2).
\end{align*}

The \emph{modulus of a curve family} $\Gamma$ in $\R^n$ is \cite[(7.1), p. 104]{hkv}
\begin{align*}
\M(\Gamma)=\inf_{\rho\in\mathcal{F}(\Gamma)}\int_{\R^n}\rho^ndm,
\end{align*}
where $\mathcal{F}(\Gamma)$ consists of all non-negative Borel-measurable functions $\rho:\R^n\to\overline{\R}^n$ such that $\int_\gamma\rho ds\geq1$ for each locally rectifiable curve $\gamma\in\Gamma$ and $m$ stands for the $n$-dimensional Lebesgue measure. Denote the family of all curves joining two non-empty sets $E$ and $F$ in a domain $G$ by $\Delta(E,F;G)$. Now, for the annular ring $D=\overline{B}^n(0,b)\backslash B^n(0,a)$ with $0<a<b$, it holds that \cite[(7.4), p. 107]{hkv}
\begin{align*}
\M(\Delta(S^{n-1}(0,a),S^{n-1}(0,b);D))=\omega_{n-1}\left(\log\frac{b}{a}\right)^{1-n}.
\end{align*}
Any domain $G$ and its compact subset $F\subset G$ form a \emph{condenser} $(G,F)$ and the \emph{capacity} of this condenser is \cite[Thm 9.6, p. 152]{hkv}
\begin{align*}
\capa(G,F)=\M(\Delta(F,\partial G;G)).
\end{align*}
A compact set $E$ in $\mathbb{R}^n$ is said to be of capacity zero, denoted ${\rm cap} (E)=0$, if there exists a bounded open set $A$ containing $E$ with the capacity of the pair $(A, E)$ is equal to zero. A compact set $E\subset \overline{\mathbb{R}}^n$, $E\neq \overline{\mathbb{R}}^n$ is said to be of capacity zero if $E$ can be mapped
by a M\"obius transformation onto a bounded set of capacity zero. If ${\rm cap} (E)=0$ does not hold, we express it as ${\rm cap}(E) > 0$.

Define next two decreasing homeomorphisms called the Gr\"otzsch capacity $\gamma_n:(1,\infty)\to(0,\infty)$ and the Teichm\"uller capacity $\tau_n:(0,\infty)\to(0,\infty)$ with the following formula \cite[(7.17), p. 121]{hkv}:
\begin{equation*}
\gamma_n(s)=\M(\Delta(\overline{\B}^n,[se_1,\infty];\R^n)),\quad s>1
\end{equation*}
and
\begin{equation*}
\tau_n(s)=\M(\Delta([-e_1,0],[se_1,\infty];\R^n)),\quad s>0,
\end{equation*}
where $e_1=(1,0,\ldots,0)$ is a unit vector in $\R^n$. For $s>1$, $\gamma_n(s)=2^{n-1}\tau_n(s^2-1)$. If $n=2$, the capacities can be computed with the formula \cite[(7.18), p. 122]{hkv},
\begin{equation}\label{gam2}
\gamma_2(1/r)=\frac{2\pi}{\mu(r)},\quad 0<r<1,
\end{equation}
where
\begin{equation*}
\mu(r)=\frac{\pi}{2}\frac{\K(\sqrt{1-r^2})}{\K(r)},\quad{\text{and}}\quad
\K(r)=\int^1_0 \frac{dx}{\sqrt{(1-x^2)(1-r^2x^2)}}.
\end{equation*}
Moreover, for $s>1$ \cite[(7.20)]{hkv}
\begin{equation}\label{gam3}
\mu(1/s)\mu\left(\frac{s-1}{s+1}\right)= \frac{\pi^2}{2},
\end{equation}
and for $0<r<1$ (see, \cite[5.30]{avv})
\begin{equation}
    {\rm arth}\, \sqrt[4]{r'}<\mu(r)<\frac{\pi^2}{4\, {\rm arth}\sqrt[4]{r}},\quad r'=\sqrt{1-r^2}.
\end{equation}
For approximation of $\mu(r)$, see \cite{krv-Landen}.
The Gr\"otzsch capacity has the following well-known estimates \cite[Thm 9.17(2), p. 160]{hkv}
\begin{align}\label{gamEst}
2^{n-1}c_n\log\left(\frac{s+1}{s-1}\right)\leq\gamma_n(s)\leq2^{n-1}c_n\mu\left(\frac{s-1}{s+1}\right)<2^{n-1}c_n\log\left(4\frac{s+1}{s-1}\right).
\end{align}
where the second inequality holds as an identity for $n=2$ by \eqref{gam3}.

A homeomorphism $f:G\to G'$ between two domains $G,G'\subset\R^n$, $n\geq2$, is called \emph{$K$-quasiconformal} with some constant $K\geq1$, if the two-sided inequality
\begin{align*}
\M(\Gamma)/K\leq\M(f(\Gamma))\leq K\M(\Gamma)
\end{align*}
holds for every curve family $\Gamma$ in $G$.

J. Ferrand \cite{lf} posed the question whether $\mu_G$ bi-Lipschitz homeomorphisms are quasiconformal. This question was studied in \cite{fmv} where it was
proved that these mappings are locally H\"older continuous,
but they need not be quasiconformal and thus Ferrand's question was solved
in the negative in \cite{fmv}.
Furthermore, the radial mapping $g:\B^n\to\B^n$ defined as
\begin{equation*}
    g(x)=\left\{
  \begin{array}{ll}
    |x|^{\alpha-1}x, & \hbox{$x\in\B^n\backslash\{0\}$;} \\\\
    0, & \hbox{$x=0$,}
  \end{array}
\right.
\end{equation*}
is quasiconformal, as noted in \cite[16.2]{vbook}, and H\"older continuous
but not Lipschitz with respect to the Euclidean metric. Thus, the bi-Lipschitz condition of the modulus metric under $K$-quasiconformal mappings does not imply the same property for the Euclidean metric.


\section{Upper and lower bounds for modulus metric}\label{sec3-ul bounds}

Recently, the modulus metric $\mu_G$ has been studied in \cite{svz} where a characterization of its completeness is given. Also in \cite{svz}, a new lower bound for $\mu_G$ was found in terms of the M\"obius invariant metric $\delta_G$ (also, $\delta_G$ is called the Seittenranta metric \cite[p. 75 \& p. 199]{hkv}): If the boundary $\partial G$ is uniformly perfect, then
\begin{align}\label{mudeltalb}
\mu_G(x,y)\geq c\delta_G(x,y),
\end{align}
where the constant $c$ only depends on the dimension $n$ and the parameters of the uniform perfectness. The estimates in \eqref{gamEst} also yield
\begin{align}\label{muboundsinD}
2^{n-1}c_n\rho_D(x,y)\leq\mu_D(x,y)\leq2^{n-1}c_n\mu(1/e^{\rho_D(x,y)})<2^{n-1}c_n(\rho_D(x,y)+\log4),
\end{align}
where $D\in\{\uhp^n,\B^n\}$. A similar inequality can be also written for $\lambda_D$:
\begin{equation*}
\begin{aligned}
c_n\log(t)
\leq\lambda_D(x,y)
\leq\frac{c_n}{2}\mu\left(t^{-2}\right)
<c_n\log(2t),\quad\text{with}\quad
t=\frac{e^{\rho_D(x,y)/2}+1}{e^{\rho_D(x,y)/2}-1}.
\end{aligned}
\end{equation*}
Because $\rho_{\B^n}=\delta_{\B^n}$, the lower bound in \eqref{muboundsinD} is compatible with \eqref{mudeltalb} up to a constant factor, but it can be still improved for small values of the hyperbolic distance $\rho_{\B^2}(x,y)$ in the two-dimensional case. Note that, for $n=2$, it follows from the inequality \eqref{muboundsinD} that
\begin{align}\label{twoDimsMubounds}
\frac{4}{\pi}\rho_D(x,y)\leq\mu_D(x,y)\leq\frac{4}{\pi}(\rho_D(x,y)+\log4),
\end{align}
and consider the following preliminary result:


\begin{lemma}\label{lem_pdec}
For all $t>0$ and $p>0$, the expression $[{\rm arth}\left(({\rm th}\,t)^{1/p}\right)]^p$ is strictly increasing with respect to $p$.
\end{lemma}
\begin{proof}
Since we are interested in the expression above with respect to $p$ only, we can substitute $u={\rm th}\,t$ and study $[{\rm arth}\, (u^{1/p})]^p$ for $0<u<1$ and $p>0$ instead.
By differentiation,
\begin{align}
\frac{\partial}{\partial p} [{\rm arth}\, (u^{1/p})]^p&=[{\rm arth}\, (u^{1/p})]^p\left(\log({\rm arth}(u^{1/p}))-\frac{u^{1/p}\log(u)}{p(1-u^{2/p}){\rm arth}(u^{1/p})}\right)>0\label{deriu}\\
&\Leftrightarrow
{\rm arth}(u^{1/p})\log({\rm arth}(u^{1/p}))-\frac{u^{1/p}\log(u)}{p(1-u^{2/p})}>0\nonumber\\
&\Leftrightarrow{\rm arth}(y)\log({\rm arth}(y))-\frac{y\log(y)}{1-y^2}>0\nonumber,
\end{align}
where $0<y<1$.
Again, by differentiation,
\begin{align}
\frac{d}{d y}\left({\rm arth}(y)\log({\rm arth}(y))-\frac{y\log(y)}{1-y^2}\right)&=\frac{1}{(1-y^2)^2}((1-y^2)\log({\rm arth}(y))-(1+y^2)\log(y))>0\label{deriy}\\
&\Leftrightarrow(1-y^2)\log({\rm arth}(y))-(1+y^2)\log(y)>0\nonumber\\
&\Leftrightarrow\frac{\log({\rm arth}(y))}{\log(y)}-\frac{1+y^2}{1-y^2}<0\nonumber,
\end{align}
provided that $0<y<1$.
Since $\log({\rm arth}(y))/\log(y)$ is strictly decreasing on $(0,1)$ and $(1+y^2)/(1-y^2)$ is strictly increasing on $(0,1)$, their difference is strictly decreasing on $(0,1)$ and the following holds:
\begin{align*}
\frac{\log({\rm arth}(y))}{\log(y)}-\frac{1+y^2}{1-y^2}< \lim_{y\to0+}\left(\frac{\log({\rm arth}(y))}{\log(y)}-\frac{1+y^2}{1-y^2}\right)=1-1=0.
\end{align*}
Consequently, the derivative \eqref{deriy} is positive on $(0,1)$, the differentiated expression in \eqref{deriy} is therefore strictly increasing on this interval, and we have the following lower bound:
\begin{align*}
{\rm arth}(y)\log({\rm arth}(y))-\frac{y\log(y)}{1-y^2}> \lim_{y\to0+}\left({\rm arth}(y)\log({\rm arth}(y))-\frac{y\log(y)}{1-y^2}\right)=0.
\end{align*}
Thus, the derivative \eqref{deriu} is positive for $0<u<1$ and $p>0$ and the original expression $\left[{\rm arth}(u^{1/p})\right]^p$ is strictly increasing with respect to $p>0$, from which our lemma follows.
\end{proof}

\begin{corollary}\label{cor_muLowBound}
For all $x,y\in D\in\{\uhp^2,\B^2\}$, the inequality
\begin{align*}
\mu_D(x,y)
\geq\frac{8}{\pi\sqrt[4]{2}}\rho_D(x,y)^{1/4}
\end{align*}
holds.
\end{corollary}
\begin{proof}
If $x=y$, the result holds trivially, so let us assume that $x\neq y$ below. By Lemma \ref{lem_pdec}, the expression $[{\rm arth}\left(({\rm th}\,t)^{1/p}\right)]^p$ is strictly increasing with respect to $p$ for all $t>0$ and $p>0$, so
\begin{align*}
[{\rm arth}\left(({\rm th}\,t)^{1/4}\right)]^4>{\rm arth}({\rm th}(t))=t
\Leftrightarrow
{\rm arth}({\rm th}(t)^{1/4})>t^{1/4}.
\end{align*}
Combining this result to Theorem \ref{thm_muLamB}, the formula \eqref{gam2} and \cite[(5.29)]{avv}, we have
\begin{equation*}
\mu_D(x,y)=\frac{2\pi}{\mu({\rm th}(\rho_D(x,y)/2))}
>\frac{8}{\pi}{\rm arth}(({\rm th}(\rho_D(x,y)/2)^{1/4})
>\frac{8}{\pi}(\rho_D(x,y)/2)^{1/4}.
\end{equation*}
\end{proof}

\begin{remark}
Corollary \ref{cor_muLowBound} gives a better lower bound for the modulus metric than the inequality \eqref{twoDimsMubounds} if and only if the hyperbolic distance $\rho_{\B^2}(x,y)$ is less than 2. See Figure \ref{fig-mulowerb} for more details. We note that in Figure \ref{Fig01}, $\rho_{\B^2}(x,0)<2$ if and only if $0<x\leq 0.75$. Also, in Figure \ref{Fig02}, $\rho_{\B^2}(x,0)\geq 2$ if and only if $0.75< x<1$.
\begin{figure}[!ht]
\centering
\subfigure[]{
\includegraphics[width=2.9in]{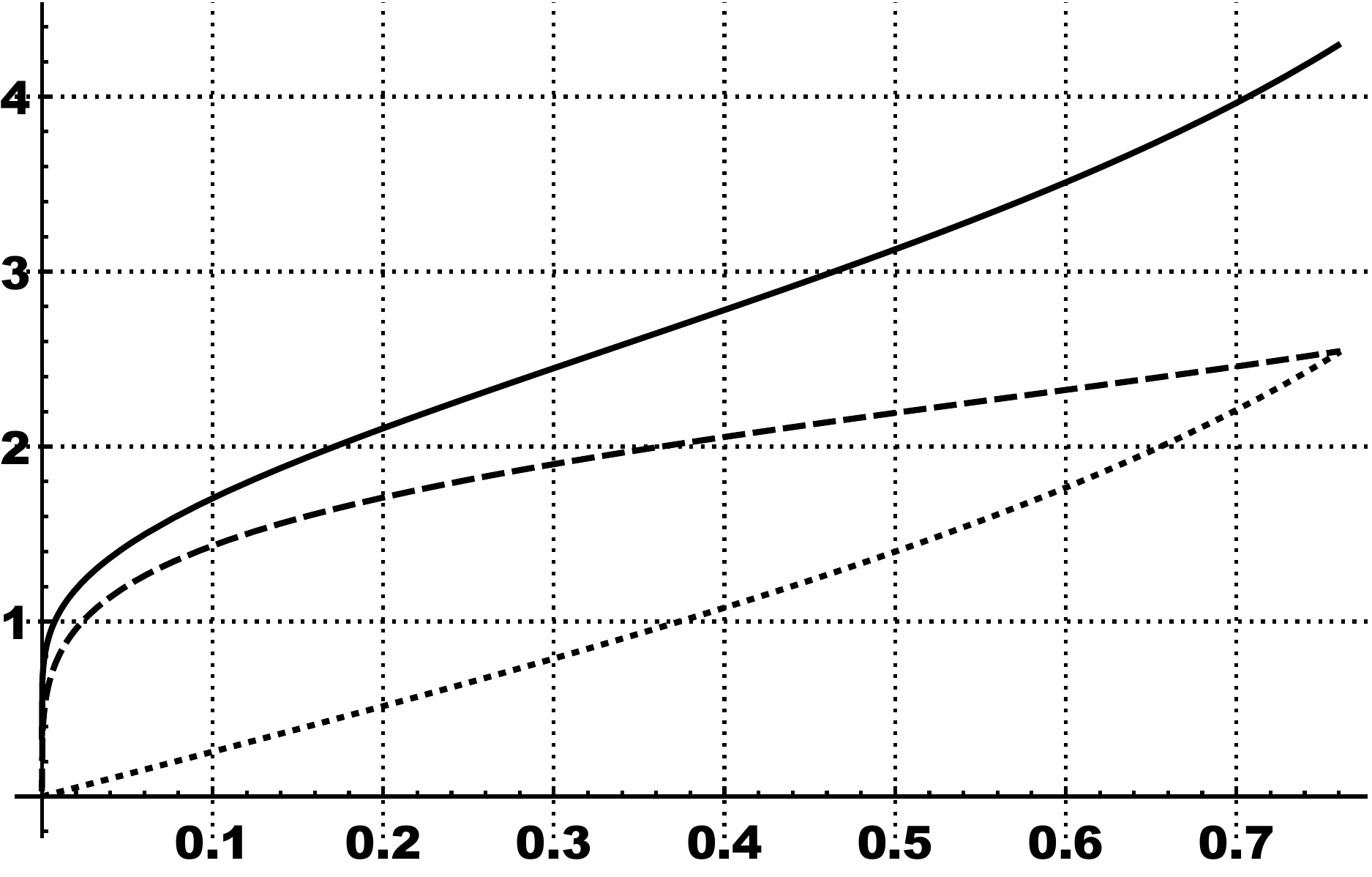}
    \label{Fig01}
}
\hspace*{7mm}
\subfigure[]{
\includegraphics[width=2.9in]{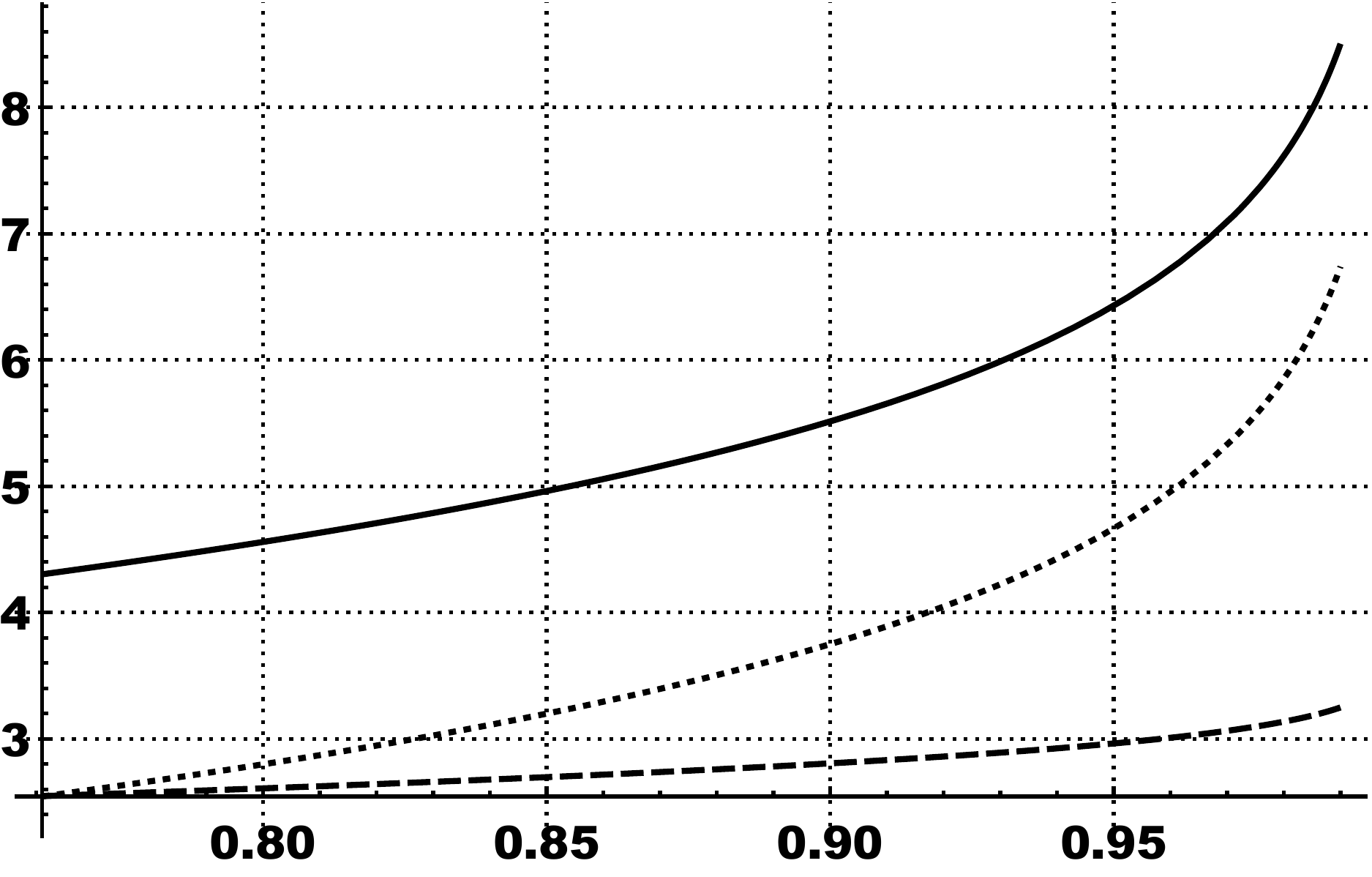}
    \label{Fig02}
}
\caption[Figures]
{\subref{Fig01}: The graph of $\mu_{\B^2}(x,0)$, its lower bound in Corollary \ref{cor_muLowBound} (dashed), and its lower bound in \eqref{twoDimsMubounds} (dotted) when $\rho_{\B^2}(x,0)<2$, where $0<x\leq 0.75$.
 \subref{Fig02}: The graph of $\mu_{\B^2}(x,0)$, its lower bound in \eqref{twoDimsMubounds} (dotted), and its lower bound in Corollary \ref{cor_muLowBound} (dashed) when $\rho_{\B^2}(x,0)\geq 2$, where $0.75<x<1$.}
 \label{fig-mulowerb}
\end{figure}
\end{remark}

In article \cite{sinb}, the Euclidean midpoint rotation was introduced
as a new way to find upper and lower bounds for the triangular ratio
metric. This metric was originally introduced in 2002 by P. H\"ast\"o \cite{h}
and recently studied in \cite{sch, fss, sqm}. In the midpoint rotation,
two distinct points $x,y\in\B^2$ are rotated around their Euclidean
midpoint $(x+y)/2$ so that the smaller angle $\nu$ between lines
$L(x,y)$ and $L(0,(x+y)/2)$ varies on the interval $[0,\pi/2]$.
See Figure \ref{fig1}. Here, we assume that $x\neq -y$, because otherwise
the midpoint $(x+y)/2$ is the origin and the hyperbolic distance
$\rho_{\B^2}(x,y)$ is invariant under rotations around the origin. As
explained in Theorem \ref{thm_emrForRho}, the hyperbolic distance of the rotated points is decreasing with respect to $\nu$ and, since
$\nu=0$ when the rotated points are collinear with the origin and
$\nu=\pi/2$ when their absolute values are equivalent, the Euclidean
midpoint rotation yields upper and lower bounds for the hyperbolic metric.

\begin{figure}[ht]
    \centering
    \begin{tikzpicture}[scale=6]
    \draw[thick] (1,0) arc (0:90:0.9);
    \draw[thick] (0.4,0.4) circle (0.3cm);
    \draw[thick] (0.187,0.613) -- (0.613,0.187);
    \draw[thick] (0,0) -- (0.9,0.9);
    \draw[thick] (0.4+0.05*1.9,0.4+0.295*1.9) -- (0.4-0.05*1.4,0.4-0.295*1.4);
    \node [black] at (0,0) {\textbullet};
    \node [black] at (0.4,0.4) {\textbullet};
    \node [black] at (0.45,0.695) {\textbullet};
    \node [black] at (0.35,0.105) {\textbullet};
    \draw[thick] (0.44,0.44) arc (45:83:0.056);
    \node[scale=1.3] at (-0.03,0.07) {$0$};
    \node[scale=1.3] at (0.41,0.74) {$x$};
    \node[scale=1.3] at (0.38,0.05) {$y$};
    \node[scale=1.3] at (0.45,0.5) {$\nu$};
    \end{tikzpicture}
    \caption{In the Euclidean midpoint rotation, two distinct points $x,y$ in the unit disk $\B^2$ are rotated around their midpoint $(x+y)/2$ so that the smaller angle $\nu$ between the lines $L(x,y)$ and $L(0,(x+y)/2)$ increases from $0$ to $\pi/2$.}
    \label{fig1}
\end{figure}
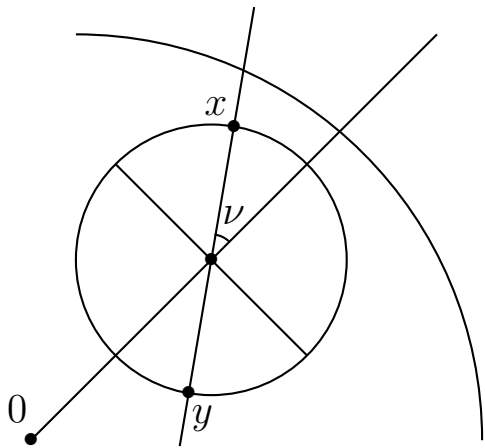

\begin{theorem}\label{thm_emrForRho}
For all distinct points $x,y\in\B^2$ such that $x\neq-y$,
the hyperbolic distance $\rho_{\B^2}(x,y)$ decreases as $x$ and $y$
are rotated around their Euclidean midpoint so that the smaller
angle between the lines $L(x,y)$ and $L(0,(x+y)/2)$ increases from $0$ to $\pi/2$. Furthermore,
\begin{align*}
\frac{2|x-y|}{\sqrt{4-8x\cdot y+(|x|^2+|y|^2)^2}}
\leq{\rm th}\frac{\rho_{\B^2}(x,y)}{2}
\leq\frac{|x-y|}{1-x\cdot y},
\end{align*}
where equality holds in the first inequality
if and only if $|x|=|y|$ and in the second inequality
if and only if $x,y$ are collinear with the origin. Note also
that ${|x-y|}/({1-x\cdot y})<1$ if and only if $|x+y|+|x-y|<2$.
\end{theorem}
\begin{proof}
Fix distinct points $x,y\in\B^2$ such that $x\neq-y$.
Denote $d=|x-y|/2$ and $k=|x+y|/2$, and note that $k,d\in(0,1)$.
Let $\nu\in[0,\pi/2]$ be the magnitude of the smaller angle between the lines $L(x,y)$ and $L(0,(x+y)/2)$. Now, we have,
\begin{align*}
{\rm th}(\rho_{\B^2}(x,y)/2)
&=\frac{|x-y|}{|1-x\overline{y}|}\\
&=\frac{2d}{|1-(k+d(\cos(\nu)+\sin(\nu)i))(k+d(-\cos(\nu)+\sin(\nu)i)|}\\
&=\frac{2d}{|1+d^2-k^2-2kd\sin(\nu)i|}
=\frac{2d}{\sqrt{(1+d^2-k^2)^2+4k^2d^2\sin^2(\nu)}}.
\end{align*}
Trivially, the quotient above is decreasing with respect to $\nu$ and it attains its minimum $2d/\sqrt{(1+d^2-k^2)^2+4k^2d^2}$ at $\nu=\pi/2$ and its maximum $2d/(1+d^2-k^2)$ with respect to $\nu$ at $\nu=0$. Given $|x-y|^2=|x|^2+|y|^2-2x\cdot y$, we can easily show that
\begin{align*}
\frac{2d}{\sqrt{(1+d^2-k^2)^2+4k^2d^2}}
&=\frac{4|x-y|}{\sqrt{(4+|x-y|^2-|x+y|^2)^2+4|x+y|^2|x-y|^2}}\\
&=\frac{2|x-y|}{\sqrt{4-8x\cdot y+(|x|^2+|y|^2)^2}}
\end{align*}
and
\begin{equation*}
\frac{2d}{1+d^2-k^2}
=\frac{4|x-y|}{4+|x-y|^2-|x+y|^2}
=\frac{|x-y|}{1-x\cdot y}.
\end{equation*}
Note that the points always stay in $\B^2$ if the rotation is done so that $\nu$ increases on the interval $[0,\pi/2]$ whereas a rotation decreasing $v$ might move one of the points outside of $\B^2$ if $k+d>1$.
\end{proof}

By \cite[7.64(26-27), p. 156]{avv},
\begin{equation}\label{ine_rhobonds}
\begin{aligned}
\frac{|x-y|}{\min\{|x-y|+\sqrt{1-|x|^2}\sqrt{1-|y|^2},1+|x||y|\}}
&\leq{\rm th}\frac{\rho_{\B^2}(x,y)}{2}\\
\leq\frac{|x-y|}{\max\{|x-y|+(1-|x|)(1-|y|),1-|x||y|\}}.
\end{aligned}
\end{equation}
Compared to these bounds, the bounds of Theorem \ref{thm_emrForRho} are better in some cases but not always. For instance, if $x=0.6+0.3i$ and $y=0.1+0.1i$, both the upper and lower bound of Theorem \ref{thm_emrForRho} are better than those of inequality \eqref{ine_rhobonds}, but the bounds of Theorem \ref{thm_emrForRho} are worse for $x=-0.7+0.7i$ and $y=0.65-0.6i$. We summarize our observations in the following table.
\vspace{0.5cm}
\begin{center}
\begin{tabular}{ |c|c|c|c|c|c| }
 \hline
 $x$ & $y$ & L.H.S Thm \ref{thm_emrForRho} & L.H.S \eqref{ine_rhobonds} & R.H.S Thm \ref{thm_emrForRho} & R.H.S \eqref{ine_rhobonds}  \\
 \hline
 $0.6+0.3i$ & $0.1+0.1i$ & {\bf 0.575624} & 0.491855 & {\bf 0.591776} & 0.594959 \\
 \hline
  $-0.7+0.7i$ & $0.65-0.6i$ & 0.997999 & {\bf 0.999183} & 0.999555 & {\bf 0.999381} \\
 \hline
\end{tabular}
\end{center}
\vspace{0.5cm}

\begin{corollary}
For all distinct points $x,y\in\B^2$ such that $x\neq-y$,
 the distance $\mu_{\B^2}(x,y)$ decreases as $x$ and $y$ are rotated
 around their Euclidean midpoint so that the smaller angle between lines $L(x,y)$ and $L(0,(x+y)/2)$ increases from $0$ to $\pi/2$. Furthermore, if $|x+y|+|x-y|<2$,
\begin{align*}
\gamma_2\left(\frac{\sqrt{4-8x\cdot y+(|x|^2+|y|^2)^2}}{2|x-y|}\right)
\leq\mu_{\B^2}(x,y)
\leq\gamma_2\left(\frac{1-x\cdot y}{|x-y|}\right),
\end{align*}
where equality holds in the first inequality if and only if $|x|=|y|$ and in the second inequality if and only if $x,y$ are collinear with the origin. If $|x+y|+|x-y|\geq2$, only the first inequality above holds.
\end{corollary}

\begin{lemma}\label{lem-not Holder}
The modulus metric defined in the unit disk is not H\"older continuous with respect to the Euclidean metric.
\end{lemma}
\begin{proof}
First, fix $w>0$ and $x\in(0,1)\subset\B^2$.
By Theorem \ref{thm_muLamB}, the formula \eqref{gam2},
and the inequality \cite[(7.21), p. 122]{hkv},
\begin{equation*}
\frac{\mu_{\B^2}(x,0)}{|x|^w}
=\frac{\gamma_2(1/{\rm th}(\rho_{\B^2}(x,0)/2))}{|x|^w}
=\frac{\gamma_2(1/|x|)}{|x|^w}
=\frac{2\pi}{|x|^w\mu(|x|)}\geq\frac{2\pi}{|x|^w U(|x|)},
\end{equation*}
where $U(r)=\log(2(1+\sqrt{1-r^2})/r)$.
Since  $\lim_{r\to0+}U(r)=\infty=\lim_{r\to0+} r^{-w}$
and $(\partial/\partial r) r^{-w}=-wr^{-w-1}>0$
for all $r>0$, it follows from L'H\^opital's rule that
\begin{equation*}
\lim_{r\to0+} r^wU(r)
=\lim_{r\to0+}\frac{U(r)}{r^{-w}}
=\lim_{r\to0+}\frac{U'(r)}{\frac{\partial}{\partial r} r^{-w}}
=\lim_{r\to0+}\frac{-1/(r\sqrt{1-r^2})}{-wr^{-w-1}}
=\lim_{r\to0+}\frac{r^w}{w\sqrt{1-r^2}}
=0.
\end{equation*}
Consequently,
\begin{align*}
\lim_{x\to0+}\frac{\mu_{\B^2}(x,0)}{|x|^w}
\geq\lim_{x\to0+}\frac{2\pi}{|x|^wU(|x|)}
=\frac{2\pi}{\lim_{x\to0+}|x|^wU(|x|)}
=\infty.
\end{align*}
Thus, the quotient $\mu_{\B^2}(x,0)/|x|^w$ approaches infinity as $x\to0+$ for all $w>0$, from which our result follows.
\end{proof}
\begin{nonsec}{\bf Proof of Theorem \ref{nonHol}.}
The proof follows from Lemma \ref{lem-not Holder}.
\hfill $\square$
\end{nonsec}
\begin{cor}\label{cor-FerNotHol}
The metric $\lambda^{-1}_{\B^2}(x,y)$ defined in the unit disk is not H\"older continuous with respect to the hyperbolic metric.
\end{cor}
\begin{proof}
The result follows from Lemma \ref{lem-not Holder} and the identity $\mu_{\B^2}(x,y)\lambda_{\B^2}(x,y)=4$, see \cite[16.7(1)]{avv}.
\end{proof}
\vspace{0.5cm}
\noindent
{\bf Acknowledgments.}
We are indebted to Prof. Matti Vuorinen for introducing this topic to
us and offering useful suggestions during the writing process. We also wish to extend our appreciation to the reviewers for the priceless feedback they provided.

\vspace{0.5cm}
\noindent
{\bf Funding.}
  The first author was financially supported by the University of Turku Graduate School, Doctoral Programme in Exact Sciences (EXACTUS). The research of the second author was supported by the Magnus Ehrnrooth Foundation.

\subsection*{Conflict of interest.}

The authors declare that they have no conflict of interest.


\def\cprime{$'$} \def\cprime{$'$} \def\cprime{$'$}
\providecommand{\bysame}{\leavevmode\hbox to3em{\hrulefill}\thinspace}
\providecommand{\MR}{\relax\ifhmode\unskip\space\fi MR }
\providecommand{\MRhref}[2]{%
  \href{http://www.ams.org/mathscinet-getitem?mr=#1}{#2}
}
\providecommand{\href}[2]{#2}



\end{document}